\documentclass[reqno,a4paper]{amsart}
\usepackage{amssymb}
\usepackage[colorlinks=true,hypertex]{hyperref}
\allowdisplaybreaks[4]
\theoremstyle{plain}
\newtheorem{thm}{Theorem}

\theoremstyle{remark}
\newtheorem{rem}{Remark}

\date{Commenced on 29 January 2009 and completed on 31 January 2009 in Melbourne; Revised on 19 March 2009 in Jiaozuo}
\date{}

\begin{document}

\title[Sharpening and generalizations of Shafer-Fink's double inequality]
{Sharpening and generalizations of Shafer-Fink's double inequality for the arc sine function}

\author[F. Qi]{Feng Qi}
\address[F. Qi]{Department of Mathematics, College of Science, Tianjin Polytechnic University, Tianjin City, 300160, China}
\email{\href{mailto: F. Qi <qifeng618@gmail.com>}{qifeng618@gmail.com}, \href{mailto: F. Qi <qifeng618@hotmail.com>}{qifeng618@hotmail.com}, \href{mailto: F. Qi <qifeng618@qq.com>}{qifeng618@qq.com}}
\urladdr{\url{http://qifeng618.spaces.live.com}}

\author[B.-N. Guo]{Bai-Ni Guo}
\address[B.-N. Guo]{School of Mathematics and Informatics, Henan Polytechnic University, Jiaozuo City, Henan Province, 454010, China}
\email{\href{mailto: B.-N. Guo <bai.ni.guo@gmail.com>}{bai.ni.guo@gmail.com}, \href{mailto: B.-N. Guo <bai.ni.guo@hotmail.com>}{bai.ni.guo@hotmail.com}}
\urladdr{\url{http://guobaini.spaces.live.com}}

\begin{abstract}
In this paper, we sharpen and generalize Shafer-Fink's double inequality for the arc sine function.
\end{abstract}



\thanks{The first author was partially supported by the China Scholarship Council}

\thanks{This paper was typeset using \AmS-\LaTeX}

\maketitle

\section{Introduction and main results}

In~\cite[p.~247, 3.4.31]{mit}, it was listed that the inequality
\begin{equation}\label{Shafer-ineq-arcsin}
\arcsin x>\frac{6\bigl(\sqrt{1+x}\,-\sqrt{1-x}\,\bigr)}{4+\sqrt{1+x}\,+\sqrt{1-x}\,} >\frac{3x}{2+\sqrt{1-x^2}\,}
\end{equation}
holds for $0<x<1$. It was also pointed out in~\cite[p.~247, 3.4.31]{mit} that these inequalities are due to R. E. Shafer, but no a related reference is cited. By now we do not know the very original source of inequalities in~\eqref{Shafer-ineq-arcsin}.
\par
In the first part of the short paper~\cite{Fink-Beograd-Univ-95}, the inequality between the very ends of~\eqref{Shafer-ineq-arcsin} was recovered and an upper bound for the arc sine function was also established as follows:
\begin{equation}\label{Shafer-Fink-ineq-arcsin}
\frac{3x}{2+\sqrt{1-x^2}\,}\le\arcsin x\le\frac{\pi x}{2+\sqrt{1-x^2}\,},\quad 0\le x\le 1.
\end{equation}
Therefore, we call~\eqref{Shafer-Fink-ineq-arcsin} the Shafer-Fink's double inequality for the arc sine function.
\par
In~\cite{Malesevic-Beograd-Univ-97}, the right-hand side inequality in~\eqref{Shafer-Fink-ineq-arcsin} was improved to
\begin{equation}\label{Malesevic-Beograd-Univ-97-ineq}
\arcsin x\le\frac{\pi x/(\pi-2)}{2/(\pi-2)+\sqrt{1-x^2}\,},\quad 0\le x\le1.
\end{equation}
\par
In~\cite{Oppeheim-Zhu-mia}, the inequality~\eqref{Malesevic-Beograd-Univ-97-ineq} was recovered and the following Shafer-Fink type inequalities were derived:
\begin{gather}\label{Zhu-mia-1}
\frac{\pi(4-\pi)x}{2/(\pi-2)+\sqrt{1-x^2}\,} \le\arcsin x,\quad 0\le x\le1;\\
\frac{(\pi/2)x}{1+\sqrt{1-x^2}\,} \le\arcsin x,\quad 0\le x\le1. \label{Zhu-mia-2}
\end{gather}
\par
Note that the lower bounds in~\eqref{Shafer-Fink-ineq-arcsin}, \eqref{Zhu-mia-1} and~\eqref{Zhu-mia-2} are not included each other.
\par
The main aim of this paper is to sharpen and generalize the above Shafer-Fink type double inequalities.
\par
Our main results can be stated as follows.

\begin{thm}\label{Shafer-Fink-thm}
For $\alpha\in\mathbb{R}$ and $x\in(0,1]$, the function
\begin{equation}
f_\alpha(x)=\Bigl(\alpha+\sqrt{1-x^2}\,\Bigr)\frac{\arcsin x}{x}
\end{equation}
is strictly
\begin{enumerate}
\item
increasing if and only if $\alpha\ge2$;
\item
decreasing if and only if $\alpha\le\frac\pi2$.
\end{enumerate}
Moreover, when $\frac\pi2<\alpha<2$, the function $f_\alpha(x)$ has a unique minimum on $(0,1)$.
\end{thm}

As straightforward consequences of Theorem~\ref{Shafer-Fink-thm}, the following double inequalities may be derived readily.

\begin{thm}\label{Shafer-Fink-thm-ineq}
If $\alpha\ge2$, the double inequality
\begin{equation}\label{Shafer-ineq-arcsin-new1}
\frac{(\alpha+1)x}{\alpha+\sqrt{1-x^2}\,}\le\arcsin x\le\frac{(\pi\alpha/2)x}{\alpha+\sqrt{1-x^2}\,}
\end{equation}
holds on $[0,1]$. If $0<\alpha\le\frac\pi2$, the inequality~\eqref{Shafer-ineq-arcsin-new1} reverses. If $\frac\pi2<\alpha<2$, then the inequality
\begin{equation}\label{max-ineq-inv}
\frac{4\bigl(1-1/\alpha^2\bigr)x}{\alpha+\sqrt{1-x^2}\,}\le \arcsin x \le\frac{\max\{\pi\alpha/2,\alpha+1\}x}{\alpha+\sqrt{1-x^2}\,}
\end{equation}
holds on $[0,1]$.
\end{thm}

\section{Remarks}

Before proving our theorems, we would like to give several remarks on them.

\begin{rem}
Letting $x=\sin t$ for $t\in\bigl[0,\frac\pi2\bigr]$ yields the restatement of Theorem~\ref{Shafer-Fink-thm-ineq} as follows:
\begin{enumerate}
\item
If $\alpha\ge2$, then
\begin{equation}\label{Shafer-oppenheim-1}
\frac{(\alpha+1)\sin t}{\alpha+\cos t}\le t\le\frac{(\pi\alpha/2)\sin t}{\alpha+\cos t}, \quad 0\le t\le\frac\pi2.
\end{equation}
\item
If $0<\alpha\le\frac\pi2$, the inequality~\eqref{Shafer-ineq-arcsin-new1} reverses.
\item
If $\frac\pi2<\alpha<2$, then
\begin{equation}\label{Shafer-oppenheim-2}
\frac{4\bigl(1-1/\alpha^2\bigr)\sin t}{\alpha+\cos t}\le t \le\frac{\max\{\pi\alpha/2,\alpha+1\}\sin t}{\alpha+\cos t}, \quad 0\le t\le\frac\pi2.
\end{equation}
\end{enumerate}
For more information on the inequalities in~\eqref{Shafer-oppenheim-1} and~\eqref{Shafer-oppenheim-2}, please refer to \cite{Oppeheim-Sin-Cos.tex} and closely-related references therein.
\end{rem}

\begin{rem}
The Shafer-Fink's double inequality~\eqref{Shafer-Fink-ineq-arcsin} is the special case $\alpha=2$ in~\eqref{Shafer-ineq-arcsin-new1}.
\end{rem}

\begin{rem}
Taking $\alpha=\frac\pi2$ in~\eqref{Shafer-ineq-arcsin-new1} gives
\begin{equation}\label{alpha=frac-pi2}
\frac{\bigl(\pi^2/4\bigr)x}{\pi/2+\sqrt{1-x^2}\,} \le\arcsin x\le \frac{(\pi/2+1)x}{\pi/2+\sqrt{1-x^2}\,},\quad 0\le x\le1.
\end{equation}
This improves the inequality~\eqref{Zhu-mia-2} and recovers the right-hand side inequality of Theorem~8 on~\cite[p.~61]{Oppeheim-Zhu-mia}.
\par
The left-hand side inequalities in~\eqref{Zhu-mia-1} and~\eqref{alpha=frac-pi2} are not included each other.
\par
The lower bound in~\eqref{alpha=frac-pi2} and those in~\eqref{Shafer-ineq-arcsin} are not included each other.
\end{rem}

\begin{rem}
Since $\frac{\pi\alpha}2=\alpha+1$ has a unique root $\alpha=\frac2{\pi-2}\in\bigl(\frac\pi2,2\bigr)$, the inequality~\eqref{Malesevic-Beograd-Univ-97-ineq} follows from taking $\alpha=\frac2{\pi-2}$ in~\eqref{max-ineq-inv}.
\end{rem}

\begin{rem}
Let
$$
h_x(\alpha)=\frac{1-1/\alpha^2}{\alpha+\sqrt{1-x^2}\,}
$$
for $\frac\pi2<\alpha<2$ and $x\in(0,1)$. Then
$$
\alpha\bigl(3-\alpha^2\bigr)<{\alpha^3 \bigl(\alpha+\sqrt{1-x^2}\,\bigr)^2}h_x'(\alpha) ={3 \alpha-\alpha^3+2 \sqrt{1-x^2}\,} <2+3 \alpha-\alpha^3.
$$
This means that
\begin{enumerate}
\item
when $\frac\pi2<\alpha\le\sqrt3\,$ the function $\alpha\mapsto h_x(\alpha)$ is increasing;
\item
when $\sqrt3\,<\alpha<2$ the function $\alpha\mapsto h_x(\alpha)$ attains its maximum
$$
\frac{4\cos^2\bigl[\frac{1}{3} \arctan\bigl(\frac{x}{\sqrt{1-x^2}\,}\bigr)\bigr]-1} {4\bigl\{2\cos\bigl[\frac{1}{3} \arctan\bigl(\frac{x}{\sqrt{1-x^2}\,}\bigr)\bigr] +\sqrt{1-x^2}\,\bigr\} \cos^2\bigl[\frac{1}{3} \arctan\bigl(\frac{x}{\sqrt{1-x^2}\,}\bigr)\bigr]}
$$
at the point
$$
2 \cos \biggl[\frac{1}{3} \arctan\biggl(\frac{x}{\sqrt{1-x^2}\,}\biggr)\biggr],\quad x\in(0,1).
$$
\end{enumerate}
Therefore, the following two sharp inequalities may be derived from the left-hand side inequality in~\eqref{max-ineq-inv} for $x\in(0,1)$:
\begin{gather}
\arcsin x > \frac{(8/3)x}{\sqrt3\,+\sqrt{1-x^2}\,},\label{arcsin-cos-1}\\
\arcsin x >\frac{x\bigl\{4\cos^2\bigl[\frac{1}{3} \arctan\bigl(\frac{x}{\sqrt{1-x^2}\,}\bigr)\bigr]-1\bigr\}} {\bigl\{2\cos\bigl[\frac{1}{3} \arctan\bigl(\frac{x}{\sqrt{1-x^2}\,}\bigr)\bigr] +\sqrt{1-x^2}\,\bigr\} \cos^2\bigl[\frac{1}{3} \arctan\bigl(\frac{x}{\sqrt{1-x^2}\,}\bigr)\bigr]}.\label{arcsin-cos-2}
\end{gather}
\par
By the famous software M\textsc{athematica}~7.0, we reveal that the inequality~\eqref{arcsin-cos-2} is better than the left-hand side inequality in~\eqref{Shafer-Fink-ineq-arcsin} and the inequalities~\eqref{Zhu-mia-1} and~\eqref{arcsin-cos-1} and that it does not include the inequality~\eqref{Zhu-mia-2} and the left-hand side inequality~\eqref{alpha=frac-pi2}.
\end{rem}

\begin{rem}
The method to prove Theorem~\ref{Shafer-Fink-thm} and Theorem~\ref{Shafer-Fink-thm-ineq} in next section have been used in \cite{Carlson-Arccos.tex, Oppeheim-Sin-Cos.tex, Shafer-ArcSin.tex, Arc-Hyperbolic-Sine.tex, Carlson-Arccos-further.tex, Shafer-Fink-ArcSin.tex-arXiv, Shafer-ArcTan.tex, Shafer-ArcTan.tex-JIA, Arc-Hyperbolic-Sine-alter.tex} and closely-related references therein.
\end{rem}

\begin{rem}
The method used in next section to prove Theorem~\ref{Shafer-Fink-thm} and Theorem~\ref{Shafer-Fink-thm-ineq} is more elementary than the one utilized in \cite{Fink-Beograd-Univ-95, Malesevic-Beograd-Univ-97, Oppeheim-Zhu-mia, Zhu-JIA-08-Shafer-Fink, Zhu-JIA-07-Shafer-Fink}.
\end{rem}

\begin{rem}
This paper is a slightly modified version of the preprint~\cite{Shafer-Fink-ArcSin.tex-arXiv}.
\end{rem}

\section{Proofs of theorems}

Now we are in a position to prove our theorems.

\begin{proof}[Proof of Theorem~\ref{Shafer-Fink-thm}]
Direct differentiation yields
\begin{align*}
f_\alpha'(x)&=\frac{\alpha+1/\sqrt{1-x^2}\,}{x^2} \biggl[\frac{x\bigl(\alpha+\sqrt{1-x^2}\,\bigr)}{1+\alpha\sqrt{1-x^2}\,}-\arcsin x\biggr]\\
&\triangleq \frac{\alpha+1/\sqrt{1-x^2}\,}{x^2} h_\alpha(x),\\
h_\alpha'(x)&=\frac{x^2\bigl(\alpha^2-2-\alpha\sqrt{1-x^2}\,\bigr)} {\bigl(1+\alpha\sqrt{1-x^2}\,\bigr)^2\sqrt{1-x^2}\,}.
\end{align*}
\par
Because
$$
\alpha^2-\alpha-2\le\alpha^2-2-\alpha\sqrt{1-x^2}\,\le\alpha^2-2
$$
on $[0,1]$, the derivative $h_\alpha'(x)$ is negative (or positive respectively) when $0<\alpha\le\sqrt2\,$ (or $\alpha\ge2$ respectively). Moreover, if $\sqrt2\,<\alpha<2$, the derivative $h_\alpha'(x)$ has a unique zero on $(0,1)$. As a result, the function $h_\alpha(x)$ is increasing (or decreasing respectively) when $\alpha\ge2$ (or $0<\alpha\le\sqrt2\,$ respectively) and has a unique minimum on $(0,1)$ when $\sqrt2\,<\alpha<2$.
It is easy to obtain that $h_\alpha(0)=0$ and $h_\alpha(1)=\alpha-\frac\pi2$. Hence,
\begin{enumerate}
\item
when $\alpha\ge2$, the function $h_\alpha(x)$ and $f_\alpha'(x)$ are positive, and so $f_\alpha(x)$ is strictly increasing on $(0,1)$;
\item
when $0<\alpha\le\sqrt2\,$, the function $h_\alpha(x)$ and $f_\alpha'(x)$ are negative, and so $f_\alpha(x)$ is strictly decreasing on $(0,1)$;
\item
when $\sqrt2\,<\alpha<2$ and $\alpha\le\frac\pi2$, the function $h_\alpha(x)$ and $f_\alpha'(x)$ are also negative, and so $f_\alpha(x)$ is also strictly decreasing on $(0,1)$;
\item
when $\sqrt2\,<\alpha<2$ and $\alpha>\frac\pi2$, the function $h_\alpha(x)$ and $f_\alpha'(x)$ have the same unique zero on $(0,1)$, and so the function $f_\alpha(x)$ has a unique minimum on $(0,1)$.
\end{enumerate}
\par
On other hand, the derivative $f_\alpha'(x)$ may be rearranged as
\begin{align*}
f_\alpha'(x)&=\frac1{x^2} \biggl[x\biggl(1+\frac{\alpha}{\sqrt{1-x^2}\,}\biggr) -\biggl(\alpha+\frac{1}{\sqrt{1-x^2}\,}\biggr)   \arcsin x\biggr]\\
&\triangleq \frac1{x^2} H_\alpha(x),\\
H_\alpha'(x)&=-\frac{x \bigl[x \bigl(\sqrt{1-x^2}\,-\alpha\bigr)+\arcsin x\bigr]}{(1-x^2)^{3/2}}.
\end{align*}
When $\alpha\le0$, the derivative $H_\alpha'(x)$ is negative, and so the function $H_\alpha(x)$ is strictly decreasing on $(0,1)$. From
$$
\lim_{x\to0^+}H_\alpha(x)=0,
$$
it follows that $H_\alpha(x)<0$ on $(0,1)$. Therefore, when $\alpha\le0$, the derivative $f_\alpha'(x)$ is negative and the function $f_\alpha(x)$ is strictly decreasing on $(0,1)$. The proof of Theorem~\ref{Shafer-Fink-thm} is complete.
\end{proof}

\begin{proof}[Proof of Theorem~\ref{Shafer-Fink-thm-ineq}]
It is easy to obtain that
$$
\lim_{x\to0^+}f_\alpha(x)=\alpha+1\quad \text{and}\quad f_\alpha(1)=\frac\pi2\alpha.
$$
From the monotonicity obtained in Theorem~\ref{Shafer-Fink-thm}, it follows that
\begin{enumerate}
\item
when $\alpha\ge2$, we have
\begin{equation}\label{Shafer-ineq-arcsin-pre}
\alpha+1<\Bigl(\alpha+\sqrt{1-x^2}\,\Bigr)\frac{\arcsin x}{x}\le\frac\pi2\alpha
\end{equation}
on $(0,1]$, which can be rewritten as the inequality~\eqref{Shafer-ineq-arcsin-new1};
\item
when $0<\alpha\le\frac\pi2$, the inequality~\eqref{Shafer-ineq-arcsin-pre} is reversed;
\item
when $\frac\pi2<\alpha<2$, we have
$$
\Bigl(\alpha+\sqrt{1-x^2}\,\Bigr)\frac{\arcsin x}{x}\le\max\biggl\{\frac\pi2\alpha,\alpha+1\biggr\}
$$
which can be rearranged as the right-hand side inequality in~\eqref{max-ineq-inv}.
\end{enumerate}
\par
On the other hand, when $\frac\pi2<\alpha<2$, the minimum point $x_0\in[0,1]$ of $f_\alpha(x)$ satisfies
$$
\frac{\arcsin x_0}{x_0}=\frac{\alpha+\sqrt{1-x_0^2}\,}{1+\alpha\sqrt{1-x_0^2}\,}.
$$
Hence, the minimum equals
$$
f_\alpha(x_0)=\frac{\bigl(\alpha+\sqrt{1-x_0^2}\,\bigr)^2}{1+\alpha\sqrt{1-x_0^2}\,} =\frac{(\alpha+u_0)^2}{1+\alpha u_0}\ge4\biggl(1-\frac1{\alpha^2}\biggr),\quad u_0\in[0,1].
$$
The right-hand side inequality in~\eqref{max-ineq-inv} follows. The proof of Theorem~\ref{Shafer-Fink-thm-ineq} is proved.
\end{proof}

\end{document}